\documentclass[reqno]{amsart}

\usepackage{amsmath,amsbsy,amsfonts,nccmath}
\usepackage[utf8]{inputenc}
\usepackage{amssymb}
\usepackage{color}

\newcommand{\dd}{\mathrm{d}}

\newtheorem{theorem}{Theorem}
\newtheorem{proposition}[theorem]{Proposition}
\newtheorem{definition}{Definition}
\newtheorem{lemma}[theorem]{Lemma}
\newtheorem{remark}[theorem]{Remark}

\begin{document}

\title[Limiting problems for eigenvalue systems]{On the limiting problems for two eigenvalue systems and variations}
\author{H. Bueno}
\address{Departmento de Matem\'atica, Universidade Federal de  Minas Gerais, 31270-901 - Belo Horizonte - MG, Brazil}
\email{hamilton@mat.ufmg.br}

\author{Aldo H. S. Medeiros}
\address{Departamento de Matem\'{a}tica,
	Universidade Federal de Viçosa, 36570-900 - Vi\c{c}osa - MG, Brazil.}
\email{aldo.medeiros@ufv.br}


\date{}
\begin{abstract}Let $\Omega$ be a bounded, smooth domain. Supposing that $\alpha(p) + \beta(p) = p$, $\forall\, p \in \left(\frac{N}{s},\infty\right)$ and $\displaystyle\lim_{p \to \infty} \alpha(p)/{p} = \theta \in (0,1)$, we consider two systems for the fractional $p$-Laplacian and a variation on the first system. The first system is the following.
$$\left\{\begin{array}{ll}
(-\Delta_p)^{s}u(x) = \lambda \alpha(p) \vert u \vert^{\alpha(p)-2} u \vert v(x_0)\vert^{\beta(p)}  & {\rm in} \ \ \Omega,\\
(-\Delta_p)^{t}v(x) = \lambda \beta(p) \left(\displaystyle\int_{\Omega}\vert u \vert^{\alpha(p)} \dd x\right) \vert v(x_0) \vert^{\beta(p)-2} v(x_0) \delta_{x_0} & {\rm in} \ \ \Omega,\\
u= v=0 & {\rm in} \ \mathbb{R}^N\setminus\Omega,\\
\end{array}\right.
$$
where $x_0$ is a point in $\overline{\Omega}$, $\lambda$ is a parameter, $0<s\leq t<1$, $\delta_x$ denotes the Dirac delta distribution centered at $x$ and $p>N/s$. 

A variation on this system is obtained by considering $x_0$ to be a point where the function $v$ attains its maximum. In this case, we denote $x_0=x_v$. 

The second one is the system
$$\left\{\begin{array}{ll}
(-\Delta_p)^{s}u(x) = \lambda \alpha(p) \vert u(x_1) \vert^{\alpha(p)-2} u(x_1) \vert v(x_2) \vert^{\beta(p)} \delta_{x_1} & {\rm in} \ \ \Omega,\\
(-\Delta_p)^{t}v(x) = \lambda \beta(p) \vert u(x_1) \vert^{\alpha(p)} \vert v(x_2) \vert^{\beta(p)-2} v(x_2) \delta_{x_2} & {\rm in} \ \ \Omega,\\
u= v=0 & {\rm in} \ \mathbb{R}^N\setminus\Omega,
\end{array}\right.
$$
where $x_1,x_2\in \Omega$ are arbitrary, $x_1\neq x_2$. Although we not consider here, a variation similar to that on the first system can be solved by practically the same method we apply. 

We obtain solutions for the systems (including the variation on the first system) and consider the asymptotic behavior of these solutions as $p\to\infty$. We prove that they converge, in the viscosity sense, to solutions of problems on $u$ and $v$. 
\end{abstract}

\maketitle

\section{Introduction}
In this paper we deal with different systems for the fractional $p$-Laplacian and study the behavior of their solutions $(u_p,v_p)$ as $p$ goes to infinity: we prove that these solutions converge, in the viscosity sense, to solutions $(u_\infty,v_\infty)$ of related systems.

Let $\Omega \subset \mathbb{R}^{N}$  be a bounded, smooth domain and, for each $x\in \Omega$, let $\delta_{x}$ be the Dirac mass concentrated at $x$. Consider also functions $\alpha,\beta\colon \left(\frac{N}{s},\infty\right) \to (1, \infty)$ satisfying
\begin{enumerate}
	\item[$(h_1)$] $\alpha(p) + \beta(p) = p$, $\forall\, p \in \left(\frac{N}{s},\infty\right)$;
	\item[$(h_2)$] $\displaystyle\lim_{p \to \infty} \frac{\alpha(p)}{p} = \theta \in (0,1)$.
\end{enumerate}

For each $p > \frac{N}{s}$, we consider the system 
\begin{equation} \label{P1}\tag{$P^1_{p}$}
\left\{\begin{array}{ll}
(-\Delta_p)^{s}u(x) = \lambda \alpha(p) \vert u \vert^{\alpha(p)-2} u \vert v(x_0)\vert^{\beta(p)}  & {\rm in} \ \ \Omega,\\
(-\Delta_p)^{t}v(x) = \lambda \beta(p) \left(\displaystyle\int_{\Omega}\vert u \vert^{\alpha(p)} \dd x\right) \vert v(x_0) \vert^{\beta(p)-2} v(x_0) \delta_{x_0} & {\rm in} \ \ \Omega,\\
u= v=0 & {\rm in} \ \mathbb{R}^N\setminus\Omega,\\
\end{array}\right.
\end{equation}
where $x_0$ is a point in $\overline{\Omega}$, $\lambda$ is a parameter, $0<s\leq t<1$ and  $(-\Delta_p)^{r}$ denotes the $r$-fractional $p$-Laplacian operator, which is defined, for any $p>1$, by
\begin{equation}
(-\Delta_p)^{r}\phi(x) = \lim_{\varepsilon \to 0} \int_{\mathbb{R}^{N}\setminus B_{\varepsilon}(x)} \frac{\vert\phi(x) - \phi(y) \vert^{p-2}(\phi(x) - \phi(y))}{\vert x- y \vert^{N+rp}} \dd x \dd y 
\end{equation}
for any $\phi \in C^{\infty}_0(\Omega)$, which is a dense subespace of $W^{r,p}_0(\Omega)$. We also recall that 
$$\big\langle (-\Delta_p)^r u, \varphi \big\rangle := \int_{\mathbb{R}^{N}}\int_{\mathbb{R}^{N}} \frac{\vert u(x) -u(y) \vert^{p-2}(u(x) - u(y))(\varphi(x) - \varphi(y))}{\vert x-y \vert^{N+rp}} \dd x \dd y
$$
is the expression of $(-\Delta_p)^r$ as an operator from $W^{r,p}_0(\Omega)$ into its dual. (The definition of the space $W^{r,p}_0(\Omega)$ will be given in the sequence.)

We first prove that, for each $p>N/s$, this system has a unique solution. Then we consider the behavior of a sequence of these solutions as $p\to \infty$ and prove that they converge uniformly to $(u_\infty,v_\infty)$, which are viscosity solutions of a related system. (Precise statements are given in the sequence.)

As a variation on system \eqref{P1}, we consider the system 
\begin{equation} \label{Pinfty}\tag{$P^1_\infty$}
\left\{\begin{array}{ll}
(-\Delta_p)^{s}u(x) = \lambda \alpha(p) \vert u \vert^{\alpha(p)-2} u \vert v(x_v)\vert^{\beta(p)}  & {\rm in} \ \ \Omega,\\
(-\Delta_p)^{t}v(x) = \lambda \beta(p) \left(\displaystyle\int_{\Omega}\vert u \vert^{\alpha(p)} \dd x\right) \vert v(x_v) \vert^{\beta(p)-2} v(x_v) \delta_{x_v} & {\rm in} \ \ \Omega,\\
u= v=0 & {\rm in} \ \mathbb{R}^N\setminus\Omega,
\end{array}\right.
\end{equation}
where $x_v$ is a maximum point of $v$ in $\overline{\Omega}$. Observe that the first equation in \eqref{Pinfty} can be replaced by $(-\Delta_p)^{s}u(x) = \lambda \alpha(p) \vert u \vert^{\alpha(p)-2} u \Vert v\Vert_{\infty}^{\beta(p)}$ in $\Omega$. To solve the above system we apply the same method used to handle problem \eqref{P1}, see Remark \ref{variation}. 

We also handle the system
\begin{equation} \label{P2}\tag{$P^2_{p}$}
\left\{\begin{array}{ll}
(-\Delta_p)^{s}u(x) = \lambda \alpha(p) \vert u(x_1) \vert^{\alpha(p)-2} u(x_1) \vert v(x_2) \vert^{\beta(p)} \delta_{x_1} & {\rm in} \ \ \Omega,\\
(-\Delta_p)^{t}v(x) = \lambda \beta(p) \vert u(x_1) \vert^{\alpha(p)} \vert v(x_2) \vert^{\beta(p)-2} v(x_2) \delta_{x_2} & {\rm in} \ \ \Omega,\\
u= v=0 & {\rm in} \ \mathbb{R}^N\setminus\Omega,
\end{array}\right.
\end{equation}
where $x_1,x_2\in\Omega$ are arbitrary points, $x_1\neq x_2$. 

Of course, we could also consider the case where $x_u$ and $x_v$ are points of maxima of $u$ and $v$, respectively, since our reasoning also solves this case.  

In Section 2--5 we handle system \eqref{P1}, while system \eqref{Pinfty} is considered in Remark \ref{variation}. Finally, in Section 6 we deal with problem \eqref{P2}.
\section{Background, setting and description of results}\label{setting}
Due to the appropriate Sobolev embedding, the solutions $(u,v)$ of both problems \eqref{P1} and \eqref{P2} must be continuous. 

Since both equations in the system have the same homogeneity, \eqref{P1} and \eqref{P2} are actually eigenvalue problems. The eigenvalue problem for the $s$-fractional $p$-Laplacian operator was studied by Lindgren and Lindqvist in the pioneering paper \cite{Lind}. Precisely, they studied the problem
\begin{equation}\label{1}
\left\{\begin{array}{ll}
(-\Delta_p)^{s}u = \lambda_1(s,p)  \vert u \vert^{p-2} u(x) & {\rm in} \ \ \Omega,\\
u=0 & {\rm in} \ \mathbb{R}^N\setminus\Omega.\\
\end{array}\right.
\end{equation}
The authors proved that the minimum of the Rayleigh quotient associated with \eqref{1}, that is,
$$\lambda_1(s,p) =\inf_{u\in W^{s,p}_0(\Omega)\setminus \{0\}} \frac{[u]_{s,p}^{p}}{\Vert u \Vert_p^p} = \frac{[\phi_p]_{s,p}^{p}}{\Vert \phi_p \Vert_p^p}.
$$
is attained by a function that does not change sign in $\Omega$.

In the case $p=\infty$ of the same paper, Lindgren and Lindqvist denoted
$$\lambda_1 (s,\infty) = \inf\left\{\frac{\left\Vert \frac{u(x) - u(y)}{\vert x-y \vert^{s}} \right\Vert_\infty}{\Vert u \Vert_\infty}\,:\,  u\in W^{s,\infty}_0(\Omega)\setminus \{0\}\right\}
$$
and showed that 
$$\lambda_1 (s,\infty) = \frac{1}{R^s} \qquad \text{and} \qquad \lim_{p \to \infty} \sqrt[p]{\lambda_1(s,p)} = \lambda_1(s,\infty),
$$
where $R = \underset{x\in\Omega}{\mathrm{max \  }}\textup{dist}(x,\mathbb{R}^{N}\setminus \Omega) = \Vert \textup{dist}(\cdot, \mathbb{R}^{N}\setminus\Omega) \Vert_\infty$.

The results obtained in relation with Eq. \eqref{1} were extended by Del Pezzo and Rossi in \cite{Del} to the case of systems of the form
\begin{equation}\label{2}
\left\{\begin{array}{ll}
(-\Delta_p)^r u(x) = \lambda \alpha(p) \vert u(x) \vert^{\alpha(p)-2} u(x) \vert v(x) \vert^{\beta(p)} & {\rm in} \ \ \Omega,\\
(-\Delta_p)^s v(x) = \lambda \beta(p) \vert u(x) \vert^{\alpha(p)} \vert v(x) \vert^{\beta(p)-2} v(x) & {\rm in} \ \ \Omega,\\
u= v=0 & {\rm in} \ \mathbb{R}^N\setminus\Omega,\\
\end{array}\right.
\end{equation}
when assumptions $(h_1)$ and $(h_2)$ are fulfilled. If for each $p \in (\frac{N}{s}, \infty)$ we denote
$$\lambda_{1,p} = \inf\left\{\frac{\frac{1}{p}[u]_{r,p}^{p} + \frac{1}{p}[v]_{s,p}^{p}}{\displaystyle\int_{\Omega} \vert u \vert^{\alpha(p)} \vert v \vert^{\beta(p)}\,\dd x}\,:\, (u,v) \in W^{s,p}(\Omega), \ \ uv \neq 0\right\}
$$
the authors showed that $\lambda_{1,p}$ is \emph{principal eigenvalue} (that is, an eigenvalue associated with an eigenfunction that does not change its sign) and 
\begin{equation}\label{lambdainfty}
\lambda_{s,p}^{\frac{1}{p}} \to \Lambda_{1,\infty} = \left[\frac{1}{R}\right]^{\theta r + (1-\theta)s}\ \ \text{as} \ \ p \to \infty.
\end{equation}

More recently, Mih\v{a}ilescu, Rossi and  Stancu-Dumitru \cite{MJRD} studied the system
\begin{equation}\label{M}
\left\{\begin{array}{ll}
-\Delta_p u(x) = \lambda \alpha(p) \vert u(x_1) \vert^{\alpha(p)-2} u(x_1) \vert v(x_2) \vert^{\beta(p)} \delta_{x_1} & {\rm in} \ \ \Omega,\\
-\Delta_pv(x) = \lambda \beta(p) \vert u(x_1) \vert^{\alpha(p)} \vert v(x_2) \vert^{\beta(p)-2} v(x_2) \delta_{x_2} & {\rm in} \ \ \Omega,\\
u= v=0 & {\rm on} \ \partial\Omega,\\
\end{array}\right.
\end{equation}
where $x_1,x_2\in \Omega$ are arbitrary points, $x_1\neq x_2$. If $x_1$ and $x_2$ are points of maxima of $u$ and $v$, respectively, using arguments like those in  \cite{GGC,GG,GGA}, it can be proved that \eqref{P2} is the limit, as $r\to \infty$, of the problem
\begin{equation}
	\left\{\begin{array}{ll}
		-\Delta_p u = \lambda \alpha(p) \Vert u \Vert_r^{\alpha(p)-r} \vert u \vert^r \Vert v \Vert_r^{\beta(p)} & {\rm in} \ \ \Omega,\\
		-\Delta_p v = \lambda \beta(p) \Vert u \Vert_r^{\alpha(p)} \Vert v \Vert_r^{\beta(p)-r} \vert v \vert^r & {\rm in} \ \ \Omega,\\
		u= v=0 & {\rm on} \ \partial\Omega,
	\end{array}\right.
\end{equation}
which can be solved by classical minimization procedures.

As in \cite{Del}, they proved that system \eqref{M} has a principal eigenvalue and studied the asymptotic behavior of the principal eigenvalues and corresponding positive eigenfunctions $u_p$ and $v_p$ as $p$ goes to infinity. \textcolor{red}{Mih\v{a}ilescu, Rossi and  Stancu-Dumitru proved that the converge to $u_\infty$ and $v_\infty$, both viscosity solutions of the equation $-\Delta_\infty w=0$ in .}

The main goal of this work is to study system \eqref{P1}. Note that this system is related to both systems \eqref{2} and \eqref{M}. In the last section of this article, we make clear that the method used to solve system \eqref{P1} also applies to system \eqref{P2}, thus generalizing system \eqref{M} from \cite{MJRD} to the fractional $p$-Laplacian operator. 

Due to the presence of the Dirac mass $\delta_x$, it is more natural to compare the present work with \cite{MJRD}. We note that the integral form of the fractional $p$-Laplacian is more difficult to handle than that of the $p$-Laplacian. Also, in \cite{MJRD}, it is valid the convergence  
$$\Vert \nabla u \Vert_{L^p(\Omega)} \to \Vert \vert \nabla u \vert \Vert_{L^\infty(\Omega)}, \ \ \text{for all} \ \ u \in W_{0}^{1,p}(\Omega)$$
in the $p$-Laplacian case, what does not happen when we are dealing with the Gagliardo semi-norm. Furthermore, a direct calculation with the distance function $\text{dist}(x, \mathbb{R}^{N}\setminus \Omega)$ shows that $\vert \nabla \text{dist}(x, \mathbb{R}^{N}\setminus \Omega) \vert = 1$, but this is not valid in our case, making more difficult to estimate the solutions of system \eqref{P2}. \textcolor{red}{Furthermore, the presence of the integral term in \eqref{P1} changes the equation that the viscosity solutions $u_\infty$ and $v_\infty$ satisfy, see Theorem \ref{thm4}.}

\textcolor{red}{On its turn, we will show that the eigenvalues of \eqref{P1} converge, as $p\to\infty$ to the same value $\Lambda_{1,\infty}$ given by \eqref{lambdainfty}, a result obtained in \cite{Del}.}

We introduce the notation used while handling problem \eqref{P1}. In the last section of this article, we consider problem \eqref{P2} and make the necessary adjustments. 

For each $0<r<1$ and  $p \in[1,\infty]$, we consider the Sobolev spaces $W^{r,p}(\Omega)$ 
$$W^{r,p}(\Omega) = \left\{u \in L^{p}(\Omega)\,:\, \int_{\Omega}\int_{\Omega}\frac{\vert u(x) - u(y) \vert^p}{\vert x-y \vert^{N+rp}} \dd x \dd y < \infty\right\}, 
$$
and also the spaces
$$W^{r,p}_0(\Omega)= \left\{u \in L^{p}(\mathbb{R}^{N})\,:\, u=0 \ \text{in} \ \ \mathbb{R}^{N}\setminus \Omega \ \text{and} \ [u]_{r,p} < \infty\right\},$$
where
$$[u]_{r,p}^p =\int_{\mathbb{R}^N}\int_{\mathbb{R}^{N}}\frac{\vert u(x) - u(y) \vert^p}{\vert x-y \vert^{N+rp}} \dd x \dd y.
$$

We recall that, for $0<s\leq t<1$ and $1<p<\infty$, there exists q constant $C>0$ depending only on $s$, $N$ and $p$ such that
$$\Vert f \Vert_{W^{s,p}(\Omega)} \leq C \Vert f \Vert_{W^{t,p}(\Omega)}, \ \ \text{for all} \ \ f \in W^{t,p}(\Omega).
$$
In particular, $W_{0}^{t,p}(\Omega) \hookrightarrow W_{0}^{s,p}(\Omega)$, for more details see \cite{guia}. So, we can consider only the space $W_0^{s,p}(\Omega)$.

For each $0<s\leq t<1$, $x_0 \in \Omega$ fixed and $p\in [1,\infty]$, we denote $X_{s,t,p}(\Omega) = W_{0}^{s,p}(\Omega) \times W_{0}^{t,p}(\Omega)$ and
$$X^{*}_{s,t,p}(\Omega) = \left\{(u,v) \in X_{s,t,p}(\Omega)\,:\,  \left(\int_{\Omega} \vert u \vert^{\alpha(p)} \dd x\right)v(x_0) \neq 0\right\}.$$

If $C_0(\overline{\Omega})$ stands for the space $\left\{ u \in C(\Omega)\,:\, u=0 \ \text{in} \  \mathbb{R}^{N}\setminus\Omega\right\}$, it is well-known that the immersion $W^{s,p}_0(\Omega) \hookrightarrow C_0(\overline{\Omega})$ is compact for any $p\in \left(\frac{N}{s},\infty\right)$. The compactness of this immersion is consequence of the following Morrey’s type inequality (see \cite{guia})
\begin{equation}\label{const}
\sup_{y \neq x} \frac{\vert u(x) - u(y)\vert }{\vert x- y \vert^{s-\frac{N}{p}}} \leq C [u]_{s,p}, \ \ \forall u \in W_0^{s,p}(\Omega),
\end{equation}
which holds whenever $p > \frac{N}{s}$. If $p$ is sufficiently large, the positive constant $C$ in \eqref{const} can be chosen uniformly with respect to $p$ (see \cite{Ferreira}, Remark 2.2).

Thus, denoting
$$X_0(\Omega) = C_0(\overline{\Omega}) \times C_0(\overline{\Omega}),
$$
we have the compact immersion 
$$X_{s,t,p}(\Omega) \hookrightarrow X_0(\Omega)
$$ 
for any $p\in\left(\frac{N}{s},\infty\right)$. 

For $p \in \left(\frac{N}{s}, \infty\right)$ and $u,v\in X^*_{s,t,p}$, we define
$$Q_{s,t,p}(u,v)=\frac{\frac{1}{p}[u]_{s,p}^{p} + \frac{1}{p}[v]_{t,p}^{p}}{ \left(\displaystyle\int_{\Omega} \vert u \vert^{\alpha(p)} \dd x\right)\vert v(x_0) \vert^{\beta(p)}}
$$
and
$$\Lambda_{1}(p) = \inf_{(u,v) \in X^{*}_{s,t,p}(\Omega)} Q_{,s,t,p}(u,v).
$$

Straightforward calculations show that 
\begin{equation}\label{eq3}
\frac{\dd}{\dd t}\bigg|_{t=0}\left(\frac{1}{p}[u+t\varphi]_{r,p}^{p}\right) = \big\langle (-\Delta_p)^ru, \varphi \big\rangle, \ \ \forall \varphi \in W_{0}^{r,p}(\Omega).
\end{equation}
If $0<m<\infty$, then
\begin{equation}\label{eq4}
\frac{\dd}{\dd t}\bigg|_{t=0}\vert (u+t\varphi)(x) \vert^{m} = m \vert u(x) \vert^{m-2}u(x)\varphi(x), \ \ \forall\, \varphi \in L^m(\Omega).
\end{equation}

We also have, for all $1<\alpha<\infty$ and $\varphi \in L^\alpha(\Omega)$,
\begin{equation}\label{eq5}
\frac{\dd}{\dd t}\bigg|_{t=0}\left(\int_{\Omega}\vert (u+t\varphi)(x) \vert^{\alpha}\dd x\right)\vert v(x_0) \vert^{\beta} = \alpha \left(\int_{\Omega}\vert u(x) \vert^{\alpha-2}u(x) \varphi(x) \dd x\right)\vert v(x_0) \vert^{\beta}. 
\end{equation}

\begin{definition}
A pair  $(u,v) \in X_{s,t,p}(\Omega)$ is a weak solution to \eqref{P1} if 
\begin{align}\label{defi1}
\left\langle (-\Delta_p)^s u, \varphi \right\rangle + \left\langle (-\Delta_p)^t v, \psi \right\rangle = &\lambda\left[\alpha(p)\vert u\vert^{\alpha(p)-2}u(x)\vert  v(x_0) \vert^{\beta(p)}\varphi(x)\right. \\
&\left. + \ \beta(p)\left(\int_{\Omega}\vert u(x) \vert^{\alpha(p)} \dd x\right)\vert v(x_0) \vert^{\beta(p)-2} v(x_0) \psi(x_0) \right]\nonumber
\end{align}
for all $(\varphi,\psi) \in X_{s,t,p}(\Omega)$.
\end{definition}

The functional at
the left-hand side of \eqref{defi1} is the Gâteaux derivative of the Fréchet differentiable functional $(u,v) \mapsto \displaystyle\frac{1}{p} [u]_{s,p}^{p} + \displaystyle\frac{1}{p} [v]_{t,p}^{p}$.
However, the functional at the right-hand side of \eqref{defi1} is merely related to the right-hand 
Gâteaux-derivative of the functional $(u,v) \mapsto \lambda \left(\displaystyle\int_{\Omega}\vert u(x) \vert^{\alpha(p)} \dd x \right) \vert v(x_0) \vert^{\beta(p)}$, thus motivating the definition of $Q_{p}$ and $\Lambda_1(p)$. It is noteworthy that minimizing that integral term is enough to minimize the whole system.

By applying minimization methods, our first result shows that the problem \eqref{P1} has a principal eigenvalue -- and therefore, a weak solution -- for each $p \in \left(\frac{N}{s},\infty\right)$. Its proof simply adapts Theorem 1 in \cite{MJRD}. We sketch the proof for the convenience of the reader in Section \ref{remarks}.
\begin{theorem}\label{thm1}
For each $p \in \left(\frac{N}{s}, \infty\right)$ we have
\begin{enumerate}
\item[$(i)$] $\Lambda_1(p) > 0$;
\item[$(ii)$] there exists $(u_p,v_p) \in X^{*}_{s,t,p}(\Omega)$ such that 
\[\Lambda_{1}(p) = Q_{s,t,p}(u_p,v_p),\]	
with $u_p,v_p > 0$ and $\left(\displaystyle\int_{\Omega} \vert u_p \vert^{\alpha(p)} \dd x\right)\vert  v_p(x_0) \vert^{\beta(p)} = 1$.
\end{enumerate}
\end{theorem}

The next step is to look for an operator that  will motivate the study of the problem \eqref{P1} as $p\to \infty$. So, for each $0<s\leq t <1$ and $p \in \left(\frac{N}{s}, \infty\right)$ we denote 
\begin{align*}S_{p} &= \left\{(u,v) \in X_{s,t,p}(\Omega)\,:\, \left(\int_{\Omega} \vert u \vert^{\alpha(p)} \dd x\right)\vert  v(x_0) \vert^{\beta(p)} = 1\right\}\\ S_{\infty} &= \left\{(u,v) \in X_{s,t,\infty}(\Omega)\,:\, \Vert u\Vert_{\infty}^{\theta}\vert  v(x_0) \vert^{1-\theta} = 1\right\},
\end{align*}
where $\theta$ was defined in ($h_2$).

Furthermore, for each $0<s\leq t<1$ and $p \in \left(\frac{N}{s}, \infty\right]$, we  define the functions $\chi_{S_{p}}\colon X_0(\Omega) \to [0,\infty]$ and $F_{p}\colon X_0(\Omega) \to [0,\infty]$ by
\begin{equation}\label{chi} 
\chi_{S_{p}}(u,v) = \left\{\begin{array}{ll}
0, & \text{if} \quad (u,v) \in S_{p};\\
\infty, & \text{otherwise}\\
\end{array}\right.
\end{equation}
and
\begin{equation}\label{Fsp} 
F_{p}(u,v) = \left\{\begin{array}{ll}
G_{p}(u,v) + \chi_{S_{p}}(u,v), & \text{if} \quad (u,v) \in X^{*}_{s,t,p}(\Omega);\\
\infty, & \text{otherwise},
\end{array}\right.
\end{equation}
with $G_{p}$ defined by
\begin{equation} 
G_{p}(u,v) = \left\{\begin{array}{ll}
Q_{s,t,p}(u,v)^{\frac{1}{p}}, &\text{if} \quad p \in (\frac{N}{s},\infty),\vspace*{.1cm}\\
\displaystyle\frac{\max\left\{\vert u \vert_{s}, \vert v \vert_{t}\right\}}{\Vert u\Vert_{\infty}^{\theta}\vert  v(x_0) \vert^{1-\theta}}, & \text{if}\quad p= \infty,\\
\end{array}\right.
\end{equation}
where, for $0 < \sigma<1$,
$$\vert u \vert_{\sigma} = \sup_{y \neq x}\frac{\vert u(x) - u(y)\vert }{\vert x- y \vert^{\sigma}}.
$$

The method we apply is known as $\Gamma$-convergence, but everything we use are the properties listed in Theorem \ref{thm2}. Once again, the next result follows from a straightforward adaptation of the proof of \cite[Theorem 2]{MJRD}.

\begin{theorem}\label{thm2} The function $F_{\infty}$ satisfy the following  properties.
\begin{enumerate}
\item [$(i)$] If $\{(u_p,v_p)\}$ is a sequence  such that  $(u_p,v_p) \to (u,v)$ in $X_0(\Omega)$, then
\[F_{\infty}(u,v) \leq \lim_{p\to \infty}\inf F_{p}(u_p,v_p).
\]
\item [$(ii)$] For each $(u,v)\in X_0(\Omega)$, there exists a sequence $\{(U_p,V_p)\}\subset X_0(\Omega)$ such that $(U_p,V_p) \to (u,v)$ in $X_0(\Omega)$ and
\[F_{\infty}(u,v) \geq \lim_{p\to \infty}\sup F_{p}(U_p,V_p).
\]
\end{enumerate}
\end{theorem}

Thus, as a consequence of Theorem \ref{thm2}-($i$), we have
$$F_{\infty}(u,v) \leq \lim_{p\to \infty}\inf F_{p}(u_p,v_p). $$ 
Applying this inequality to the solutions $(u_p,v_p)$ given by Theorem \ref{thm1}, we obtain the estimate 
\begin{equation}\label{F}
F_{\infty}(u,v) \leq \lim_{p\to \infty}\inf \Lambda_1(p)^{\frac{1}{p}} = \frac{1}{R^{s\theta + (1-\theta)t}} = \max\{\vert u_\infty \vert_s, \vert v_\infty \vert_t\},
\end{equation}
where the last equality will be shown in the proof of Theorem \ref{thm3}. As a consequence of Theorem \ref{thm2}-($ii$) and \eqref{F}, we can analyze problem \eqref{P1} as $p \to \infty$.

Therefore, considering Theorems \ref{thm1} and \ref{thm2}, we study the behavior of the eigenvalues and eigenfunctions of problem \eqref{P1} as $p\to \infty$. 

\begin{theorem}\label{thm3}
Let  $\{p_n\}$ be a sequence converging to $\infty$ and $(u_{p_n}, v_{p_n})$ the solution of \eqref{P1} given in  Theorem \ref{thm1}. Passing to a subsequence if necessary,  $\{(u_{p_n}, v_{p_n})\}_{n \in \mathbb{N}}$ converges uniformly to $(u_\infty, v_\infty) \in C_{0}^{0,s}(\overline{\Omega}) \times C_{0}^{0,t}(\overline{\Omega})$. Furthermore
\begin{enumerate}
\item [$(i)$] $u_\infty \geq 0$, $v_\infty \geq 0$ and $\| u_\infty \|^{\theta}_\infty \vert v_\infty(x_0) \vert^{1 - \theta}=1$;
\item [$(ii)$] $\displaystyle\lim_{n\to \infty} \sqrt[p_n]{\Lambda_1(p_n)} = \Lambda_{1,\infty} = \frac{1}{R^{s\theta + (1-\theta)t}}$;
\item [$(iii)$] $\max\left\{\vert u_\infty \vert_s, \vert v_\infty \vert_t \right\} = \displaystyle\frac{1}{R^{s\theta + (1-\theta)t}}.$
\end{enumerate}
\end{theorem}

As we will see in the sequence, the functions $u_\infty$ and $v_\infty$ are solutions, in the viscosity sense, of regular boundary value problems. In order to distinguish between the cases (and also to avoid a double minus sign), we change notation: for each $1<p<\infty$ we denote the $\sigma$-fractional $p$-Laplacian by $(-\Delta_p)^\sigma = -\mathcal{L}_{\sigma,p}$, where, if $1<p<\infty$ and $0<\sigma <1$,  
$$(\mathcal{L}_{\sigma,p}u)(x) := 2\int_{\mathbb{R}^N} \frac{\vert u(x) - u(y) \vert^{p-2}(u(x) - u(y))}{\vert x-y \vert^{N+\sigma p}} \dd y.
$$

As argued in \cite{Lind}, this expression appears formally as follows
\begin{align*}
\left\langle (-\Delta_p)^\sigma u, \varphi \right\rangle &= \int_{\mathbb{R}^{N}}\int_{\mathbb{R}^{N}}\frac{\vert u(x) -u(y) \vert^{p-2}(u(x) - u(y))(\varphi(x) - \varphi(y))}{\vert x-y \vert^{N+\sigma p}} \dd x \dd y\\
&=\int_{\mathbb{R}^{N}} \varphi(x) \left(\int_{\mathbb{R}^{N}}\frac{\vert u(x) - u(y) \vert^{p-2}(u(x) - u(y))}{\vert x-y \vert^{N+\sigma p}} \dd y \right) \dd x \\
&\quad -\int_{\mathbb{R}^{N}} \varphi(y) \left(\int_{\mathbb{R}^{N}}\frac{\vert u(x) - u(y) \vert^{p-2}(u(x) - u(y))}{\vert x-y \vert^{N+\sigma p}} \dd x \right) \dd y\\
&=\int_{\mathbb{R}^{N}} \varphi(x) (\mathcal{L}_{\sigma,p}u) (x) \dd x, \ \ \ \forall \varphi \in W_0^{\sigma,p}(\Omega).
\end{align*}

If $p=\infty$, we define
$$\mathcal{L}_{\sigma,\infty} = \mathcal{L}^{+}_{\sigma,\infty} + \mathcal{L}^{-}_{\sigma,\infty},
$$
where
$$(\mathcal{L}^{+}_{\sigma,\infty}u)(x) =\sup_{y \in \mathbb{R}^{N}\setminus\{x\}}\frac{u(x) - u(y)}{\vert x- y \vert^{\sigma}} \quad\text{and}\quad(\mathcal{L}^{-}_{\sigma,\infty}u)(x) = \inf_{y \in \mathbb{R}^{N}\setminus\{x\}} \frac{u(x) - u(y)}{\vert x- y \vert^{\sigma}},
$$
see Chambolle,  Lindgren and Monneau \cite{Chambolle}, where the concept was introduced, but also \cite{Lind}. Observe that, since $\mathcal{L}_{\sigma,\infty}$ is not sufficiently smooth, its solutions must be interpreted in the viscosity sense.

We recall the definition of a solution in the viscosity sense by considering the problem
\begin{equation} \label{Pvis}
\left\{\begin{array}{ll}
\mathcal{L}_{\sigma,p}u = 0 & {\rm in} \ \ \Omega,\\
u= 0 & {\rm in} \ \mathbb{R}^N\setminus\Omega,\\
\end{array}\right.
\end{equation}
for all $p \in (1,\infty]$.
\begin{definition}
Let $u \in C(\mathbb{R}^{N})$ satisfy $u=0$ in $\mathbb{R}^{N}\setminus\Omega$. The function $u$ is a \textbf{viscosity supersolution} of \eqref{Pvis} if
$$(\mathcal{L}_{\sigma,p}\varphi)(x_0) \leq 0
$$
for each pair $(x_0,\varphi) \in \Omega \times C_0^1(\mathbb{R}^N)$ such that
$$\varphi(x_0) = u(x_0) \qquad\text{and}\qquad\varphi(x) \leq u(x) \ \ \forall x \in \mathbb{R}^{N}.
$$
	
On its turn, $u$ is a \textbf{viscosity subsolution} of \eqref{Pvis} if
$$(\mathcal{L}_{\sigma,p}\varphi)(x_0)\geq 0
$$
for all pair $(x_0,\varphi) \in \Omega \times C_0^1(\mathbb{R}^N)$ such that
$$\varphi(x_0) = u(x_0) \ \ \text{e} \ \ \varphi(x) \geq u(x) \ \ \forall x \in \mathbb{R}^{N}.
$$
	
The function $u$ is a  \textbf{viscosity solution} to the problem \eqref{Pvis} if $u$ is both a viscosity super- and subsolution to problem \eqref{Pvis}.
\end{definition}

Finally, in Section \ref{proofthm4}, we prove that the solutions $u_\infty$ and $v_\infty$ given by Theorem \ref{thm3} are viscosity solutions. 
\begin{theorem}\label{thm4}
Let $1<s\leq t < 1$. Then, the functions $u_\infty$ and $v_\infty$, given by Theorem \ref{thm3}, are viscosity  solutions of the system
\begin{equation} \label{Pvis3}
\left\{\begin{array}{llll}
\max\left\{\mathcal{L}_{s,\infty} u, \mathcal{L}^{-}_{s,\infty} u - \Lambda_{1,\infty} \vert u(x) \vert^{\theta} \vert v_\infty(x_0)\vert^{1-\theta}\right\}= 0 & {\rm in} \ \ \Omega,\\
\mathcal{L}_{t,\infty} v = 0 & {\rm in} \ \ \Omega \setminus \{x_0\},\\
u = v = 0 & {\rm in} \ \mathbb{R}^N\setminus\Omega,\\
v(x_0) =v_\infty(x_0).
\end{array}
\right.
\end{equation}
\end{theorem}

\section{Some remarks on the proofs of Theorems \ref{thm1} and \ref{thm2}}\label{remarks}

Since the proofs of Theorems \ref{thm1} and \ref{thm2} are simple adaptations of that one given in \cite{MJRD}, we only sketch them for the convenience of the reader. For details, see \cite[Theorem 1 and Theorem 2]{MJRD}.
 
\noindent\textit{Sketch of proof of Theorem \ref{thm1}.} Estimating the denominator in the definition of $Q_{s,t,p}$, the inequalities of Young and Sobolev imply that $\Lambda_1>0$. By defining 
\begin{align*}U_n(x) &= \frac{ u_n(x)}{\left(\displaystyle \int_{\Omega} \vert u_n \vert^{\alpha(p)} \dd x\right)^{\frac{1}{p}} \vert v_n(x_0)\vert^{\frac{\beta(p)}{p}}}
\intertext{and}
V_n(x) &=\frac{v_n(x)}{\left(\displaystyle \int_{\Omega} \vert u_n \vert^{\alpha(p)} \dd x\right)^{\frac{1}{p}}\vert v_n(x_0)\vert^{\frac{\beta(p)}{p}}},
\end{align*}
we have $(U_n,V_n) \in X_{s,p}(\Omega)$ satisfy $
\left(\displaystyle\int_\Omega \vert U_n(x) \vert^{\alpha(p)}\dd x\right) |V_n(x_0)|^{\beta(p)} = 1$. Furthermore,
$$\lim_{n \to \infty} Q_{s,t,p}(U_n,V_n) =\lim_{n \to \infty} Q_{s,t,p}(u_n,v_n) = \Lambda_{1}(s,p),
$$
guaranteeing the existence of $u_p,v_p\in W^{s,p}(\Omega)$ such that 
\begin{equation*}
	\left(\int_{\Omega} \vert u_p \vert^{\alpha(p)} \dd x\right)\vert  v_p(x_0) \vert^{\beta(p)} = 1.
\end{equation*} and
$$Q_{s,t,p}(u_p, u_p) = \Lambda_{1}(p).
$$

For any $(\phi,\psi)\in X_{s,t,p}(\Omega)$, considering
$$g(t) =Q_{s,t,p}(u_p + t\phi,v_p + t\psi),$$ 
it follows the existence of $t_0>0$ such that $g(t)>g(0)=\Lambda_1(p)$. Since $g\in C^1((-t_0,t_0),\mathbb{R})$m we have $g'(0) = 0$, from what follows that $(u_p,v_p)$ is a weak solution to system \eqref{P1}. An argument similar \cite[Lemma 22]{Lind} proves that $u_p > 0 $ and $v_p > 0$ in $\Omega$,
showing that $\Lambda_{1}(s,p)$ is a principal eigenvalue to system \eqref{P1}. \hfill$\Box$\vspace*{.3cm}

\noindent\textit{Sketch of proof of Theorem \ref{thm2}.} In order to prove ($i$), suppose that $(u_p,v_p)\to (u,v)\in X_0(\Omega)$. Passing to a subsequence, we assume that $\displaystyle\lim_{p\to\infty}F_p(u_p,v_p)=\displaystyle\liminf_{p\to\infty}F_p(u_p,v_p)$. It is not difficult to discard the case $(u,v) \notin X^{*}_{s,t,\infty}(\Omega) \cap S_{\infty}$. So, we consider the case $(u,v) \in X^*_{s,t,\infty}(\Omega) \cap S_{\infty}$, which implies $\Vert u \Vert_\infty^{\theta}\vert  v(x_0) \vert^{1-\theta} = 1$. We can assume that $F_{p}(u_p,v_p) \leq C<\infty$, since otherwise ($i$) is valid. So, for $p$ large enough, we have $(u_p,v_p) \in S_{p}$ and, if $k > \frac{N}{s}$, then 

\begin{fleqn}
\begin{align*}
\left(\int_{\Omega}\int_{\Omega}\frac{\vert u_{p}(x) - u_p(y) \vert^k}{\vert x-y \vert^{\left(\frac{N}{p} + s\right)k}} + \frac{\vert v_{p}(x) - v_p(y) \vert^k }{\vert x-y \vert^{\left(\frac{N}{p} + t\right)k}} \dd x \dd  y\right)^{\frac{1}{k}}
\end{align*}
\end{fleqn} 
$$\leq 2^{\frac{1}{k}}\vert \Omega \vert^{2\left(\frac{1}{k} - \frac{1}{p}\right)} p^{\frac{1}{p}} \left[\frac{1}{p}[u_p]_{s,p}^p + \frac{1}{p} [v_p]_{t,p}^p\right]^{\frac{1}{p}}.
$$
Thus, 
\begin{align*}
F_{p}(u_p,v_p) &= Q_{s,t,p}(u_p,v_p) = \left[\frac{1}{p}[u_p]_{s,p}^p + \frac{1}{p} [v_p]_{t,p}^p\right]^{\frac{1}{p}}\\
&\geq 2^{-\frac{1}{k}}\vert \Omega \vert^{2\left(\frac{1}{p} - \frac{1}{k}\right)} p^{-\frac{1}{p}}\left(\int_{\Omega} \int_{\Omega} \frac{\vert u_{p}(x) - u_p(y) \vert^k}{\vert x-y \vert^{\left(\frac{N}{p} + s\right)k}} + \frac{\vert v_{p}(x) - v_p(y) \vert^k }{\vert x-y \vert^{\left(\frac{N}{p} + t\right)k}} \dd x \dd  y\right)^{\frac{1}{k}}.
\end{align*}
As $p\to \infty$, results from the uniform convergence and Fatou's Lemma that
\begin{align*}
\liminf_{p\to \infty} F_{p}(u_p,v_p) 
& \geq 2^{-\frac{1}{k}} \vert \Omega \vert^{-\frac{2}{k}} \left(\int_{\Omega}\int_{\Omega}\frac{\vert u(x) - u(y) \vert^k}{\vert x-y \vert^{sk}} + \frac{\vert v(x) - v(y) \vert^k }{\vert x-y \vert^{tk}} \dd x \dd  y\right)^{\frac{1}{k}}.
\end{align*}
Making $k \to \infty$, we obtain
\begin{align}\label{limit}
\liminf_{p\to \infty} F_{p}(u_p,v_p) &\geq \max\left\{\vert u\vert_s, \vert v \vert_t\right\} = F_{\infty}(u,v),
\end{align}
concluding the proof of ($i$).

Now we deal with the second claim. Take any $(u,v) \in X_0(\Omega)$ and initially suppose that $(u,v) \notin X^{*}_{s,t,\infty}(\Omega) \cap S_{\infty}$. Then $F_{s,\infty}(u,v)=\infty$.  Consider then a sequence of values $p\to\infty$ and, for any $p\in \left(\frac{N}{s}, \infty\right)$ in the sequence, define $u_p:= u$ and $v_p:= v$. Of course we have $(u_p,v_p) \to (u,v)$ as $p\to\infty$ in $X_0(\Omega)$. It is not difficult to discard the cases $\left(\displaystyle\int_{\Omega} \vert u_p \vert^{\alpha(p)} \dd x\right)\vert  v_p(x_0) \vert^{\beta(p)}\neq 1$. If, however, $(u,v) \in X^{*}_{s,t,\infty}(\Omega) \cap S_{\infty}$, consider then a sequence of values $p\to\infty$ and, for any $p\in \left(\frac{N}{s}, \infty\right)$ in the sequence, define 
$$U_p(x) = \frac{u(x)}{\left(\displaystyle \int_{\Omega} \vert u \vert^{\alpha(p)} \dd x\right)^{\frac{1}{p}} \vert v(x_0)\vert^{\frac{1}{p}}} \qquad\text{and} \qquad V_p(x) =\frac{v(x)}{\left(\displaystyle \int_{\Omega} \vert u \vert^{\alpha(p)} \dd x\right)^{\frac{1}{p}}\vert v(x_0)\vert^{\frac{\beta(p)}{p}}}.$$

Then $(U_p,V_p) \in S_{p}$ and 
\begin{align*}
\limsup_{p\to \infty} F_{p}(U_p,V_p) = \max\bigg\{\vert u \vert_s, \vert v \vert_t\bigg\} = F_{\infty}(u,v),
\end{align*}
completing the proof of ($ii$). $\hfill\Box$

\section{Proof of Theorem \ref{thm3}}
Let us denote
\[R =\max_{x \in \overline{\Omega}} \text{dist}(x, \mathbb{R}^{N}\setminus \Omega) = \Vert \text{dist}(.,\mathbb{R}^{N}\setminus \Omega) \Vert_{L^{\infty}(\Omega)}.
\]
For a fixed $x_1 \in \Omega$ we consider the functions  $\phi_{R}\colon \overline{B_{R}(x_1)} \rightarrow [0,R]$ and $\psi_{R}\colon \overline{B_{R}(x_0)} \rightarrow [0,R]$ given by 
$$\phi_{R}(x) = R^{(\theta-1)t - s\theta}\left(R-\vert x-x_1 \vert \right)^{s}_{+} \quad \text{and} \quad\psi_R(x) = R^{(\theta-1)t - s\theta}\left(R-\vert x-x_0 \vert \right)^{t}_{+}.
$$

Of course we have $\phi_R \in C_0^{0,s}(\overline{B_{R}(x_1)}$ and $\psi_R \in C_0^{0,s}(\overline{B_{R}(x_0)}$. Furthermore,
$$\Vert \phi_R \Vert_{\infty} = R^{(\theta-1)(t-s)}, \quad\vert \psi_R(x_0)| = R^{\theta(t-s)} \quad \text{and} \quad \vert\phi_R \vert_s = \vert \psi_R \vert_s = R^{(\theta-1)t - s\theta}.
$$

We can extend $\phi_R$ and $\psi_R$ to $\overline{\Omega}$ by putting $\phi_R = 0$ in $\mathbb{R}^{N}\setminus \overline{B_{R}(x_1)}$ and $\psi_R = 0$ in $\mathbb{R}^{N}\setminus \overline{B_{R}(x_0)}$ to that $\phi_R, \psi_R \in C_0^{0,s}(\overline{\Omega})$, maintaining its $s$-Hölder norm. Additionally, we still have $\phi_R, \psi_R \in W_{0}^{1,m}(\Omega) \hookrightarrow W_{0}^{s,m}(\Omega)$ for all $s \in (0,1)$ and $m \geq 1$. For details, see \cite{GGA,Lind}. 

\begin{lemma}\label{prop1}
For any fixed $0<s\leq t<1$ we have
$$\Lambda_{1,\infty} = \inf_{(u,v) \in X_{s,t,\infty}^*(\Omega)} \frac{\max\big\{\vert u \vert_s, \vert v \vert_t \big\}}{\Vert u \Vert_{\infty}^{\theta} \vert  v(x_0) \vert_{\infty}^{1-\theta}} = \frac{1}{R^{s\theta + (1-\theta)t}}.
$$
\end{lemma}
\begin{proof}We note that we have
$$\Vert \phi_R \Vert_{\infty}^{\theta} \vert \psi_R(x_0) \vert^{1-\theta} = R^{\theta(\theta-1)(t-s) + \theta(1-\theta)(t-s)} = 1
$$
and therefore
$$\Lambda_{1,\infty} = \inf_{(u,v) \in X_{s,t,\infty}^*(\Omega)}\frac{\max\big\{\vert u \vert_s, \vert v \vert_t \big\}}{\Vert u \Vert_{\infty}^{\theta} \vert  v(x_0) \vert_{\infty}^{1-\theta}} \leq \frac{\max\big\{\vert \phi_R \vert_s, \vert \psi_R \vert_t \big\}}{\Vert \phi_R \Vert_{\infty}^{\theta} \vert  \psi_R(x_0) \vert_{\infty}^{1-\theta}} = \frac{1}{R^{s\theta + (1-\theta)t}}.
$$
	
Also note that, given $(u,v) \in X_{s,t,p}^{*}(\Omega)$, then $u = 0 = v$ in $\overline{\Omega}$. Since $u$ is continuous, there exists $x_1 \in \overline{\Omega}$ such that
$$\Vert u \Vert_{\infty} = \vert u(x_1) \vert.
$$
	
The compactness of $\overline{\Omega}$ guarantees the existence of $y_{x_0}, y_{x_1} \in \partial \Omega$ such that
$$\vert x_0 - y_{x_0} \vert = \text{dist}(x_0, \mathbb{R}^{N}\setminus \Omega) \quad \text{and} \quad  \vert x_1 - y_{x_1} \vert = \text{dist}(x_1, \mathbb{R}^{N}\setminus\Omega).
$$
	
Thus, since $u(y_{x_1}) = v(y_{x_0}) = 0$, it follows
$$\Vert u \Vert_{\infty}^{\theta} = \vert u(x_1) - u(y_{x_1}) \vert^{\theta} \leq \vert u \vert_s^{\theta}\vert x_1 - y_{x_1} \vert^{s\theta} \leq \vert u \vert_s^{\theta}\,R^{s\theta}.
$$
On the other hand,
$$\vert v(x_0) \vert^{1-\theta} = \vert v(x_0) - v(y_{x_0}) \vert^{1-\theta} \leq \vert v \vert_t^{1-\theta}\vert x_0 - y_{x_0} \vert^{t(1-\theta)} \leq \vert v \vert_t^{1-\theta}\,R^{t(1-\theta)}.
$$
	
So, for any $(u,v) \in X^{*}_{s,t,p}(\Omega)$, we have
\begin{align*}\frac{1}{R^{s\theta + t(1-\theta)}} = \frac{1}{R^{s\theta}\,R^{(1-\theta)t}} &\leq \frac{\vert u \vert_s^{\theta} \vert v \vert_t^{1-\theta}}{\Vert u \Vert_{\infty}^\theta \vert  v(x_0) \vert^{1-\theta}} \leq \frac{\left(\max\big\{\vert u \vert_s, \vert v \vert_t \big\}\right)^{\theta} \left(\max\big\{\vert u \vert_s, \vert v \vert_t \big\}\right)^{1-\theta}}       {\Vert u \Vert_{\infty}^{\theta} \vert  v(x_0) \vert^{1-\theta}}\\& = \frac{\max\big\{\vert u \vert_s, \vert v \vert_t \big\}} {\Vert u \Vert_{\infty}^{\theta} \vert  v(x_0) \vert^{1-\theta}}.
\end{align*}
	
Therefore, 
$$\Lambda_{1,\infty} = \inf_{(u,v) \in X_{s,t,\infty}^*(\Omega)}\frac{\max\big\{\vert u \vert_s, \vert v \vert_t \big\}}{\Vert u \Vert_{\infty}^{\theta} \vert  v(x_0) \vert_{\infty}^{1-\theta}} \geq \frac{1}{R^{s\theta + (1-\theta)t}},
$$
concluding the proof.
\end{proof}

The next result is pivotal in our analysis of the asymptotic behavior of solutions in problems driven by the fractional $p$-Laplacian. 
\begin{lemma}\label{lemmaasymp}
Let $u \in C_0^{0,\sigma}(\overline{\Omega})$ be extended as zero outside $\Omega$. If $u \in W^{\sigma,q}(\Omega)$ for some $q>1$, then $u \in W_{0}^{\sigma,p}(\Omega)$ for all $p\geq q$ and
$$\lim_{p\to \infty} [u]_{\sigma,p} = \vert u \vert_\sigma.
$$
\end{lemma}
The proof of Lemma \ref{lemmaasymp} can be found in \cite[Lemma 7]{GGR}.

\noindent\textit{Proof of Theorem \ref{thm3}.} Of course we have
$$\Lambda_1(p_n) \leq \frac{\frac{1}{p_n}[\phi_R]_{s,p_n}^{p_n} + \frac{1}{p_n}[\psi_R]_{t,p_n}^{p_n}}{\displaystyle\int_{\Omega} \left(\vert \phi_R \vert^{\alpha(p_n)} \dd x \right)\vert \psi_R(x_0) \vert^{\beta(p_n)}}.
$$
Thus, 
\begin{align*}
\limsup_{n \to \infty} \sqrt[p_n]{\Lambda_1(p_n)} &\leq \limsup_{n \to \infty} \left(\frac{1}{p_n} \frac{[\phi_R]_{s,p_n}^{p_n} + [\psi_R]_{t,p_n}^{p_n}}{\displaystyle\int_{\Omega} \left(\vert \phi_R \vert^{\alpha(p_n)} \dd x \right)\vert \psi_R(x_0) \vert^{\beta(p_n)}}\right)^{\frac{1}{p_n}} \\
&\leq\limsup_{n \to \infty} \left( \left(\frac{2}{p_n}\right)^{\frac{1}{p_n}}  \frac{\max\big\{ [\phi_R]_{s,p_n}, [\psi_R]_{t,p_n}\big\}}{\displaystyle\int_{\Omega} \left(\vert \phi_R \vert^{\alpha(p_n)} \dd x \right)\vert \psi_R(x_0) \vert^{\beta(p_n)}}\right)\\
&= \frac{\max\big\{\vert \phi_R \vert_s, \vert \psi_R \vert_t\big\}}{\Vert \phi_R \Vert_{\infty}^{\theta}\vert \psi_R(x_0) \vert^{1-\theta}} \leq \frac{1}{R^{s\theta + (1-\theta)t}},
\end{align*}
proving that the sequence  $\left\{\sqrt[p_n]{\Lambda_1(p_n)}\right\}_{n \in \mathbb{N}}$ is bounded in $\mathbb{R}$, that is, there exists $M_0>0$ such that
\begin{equation}\label{equ1}
\sqrt[p_n]{\Lambda_1(p_n)} \leq M_0\quad \ \text{for all} \ \ n \in \mathbb{N}.
\end{equation}

Theorem \ref{thm1} guarantees that we can take $(u_{p_n},v_{p_n})$ so that
$$u_{p_n} > 0, \ v_{p_n} > 0 \quad  \text{and} \quad \left(\int_{\Omega} \vert u_{p_n} \vert^{\alpha(p_n)} \dd x \right) \vert v_{p_n}(x_0) \vert^{\beta(p_n)} = 1.
$$

Therefore
\begin{equation*}
\Lambda_1(p_n) = \frac{1}{p_n} [u_{p_n}]_{s,p_n}^{p_n} + \frac{1}{p_n}[v_{p_n}]_{t,p_n}^{p_n} \geq \frac{1}{p_n} \max\bigg\{[u_{p_n}]_{s,p_n}^{p_n},[v_{p_n}]_{s,p_n}^{p_n}\bigg\},
\end{equation*}
what yields
\begin{align}\label{upn}
[u_{p_n}]_{s,p_n}\leq p_n^{\frac{1}{p_n}}\sqrt[p_n]{\Lambda_1(s,p_n)}.
\end{align}

For a fixed $m_0 > \frac{N}{s}$, denoting the diameter of $\Omega$ by $\text{diam}(\Omega)$, it follows from \eqref{equ1} and \eqref{upn} that
\begin{align*}
\vert u_{p_n} \vert_{s-\frac{N}{m_0}} &= \sup_{x \neq y} \frac{\vert u_{p_n}(x) - u_{p_n}(y) \vert}{\vert x - y \vert^{s-\frac{N}{m_0}}}= \sup_{x \neq y} \frac{\vert u_{p_n}(x) - u_{p_n}(y) \vert}{\vert x - y \vert^{s-\frac{N}{p_n}}}\,\vert x - y \vert^{\frac{N}{m_0}-\frac{N}{p_n}}\\
&\leq \left(\text{diam}(\Omega)\right)^{\frac{N}{m_0}-\frac{N}{p_n}} \sup_{x \neq y} \frac{\vert u_{p_n}(x) - u_{p_n}(y) \vert}{\vert x - y \vert^{s-\frac{N}{p_n}}}\\
&\leq C \left(\text{diam}(\Omega)\right)^{\frac{N}{m_0}-\frac{N}{p_n}}\,[u_{p_n}]_{s,p_n} \\
&\leq C \left(\text{diam}(\Omega)\right)^{\frac{N}{m_0}-\frac{N}{p_n}}\,p_n^{\frac{1}{p_n}}\,\sqrt[p_n]{\Lambda_1(s,p_n)}\\
\end{align*}
the constant $C$ not depending on $p_n$. We conclude that the sequence $\{u_{p_n}\}$ is uniformly bounded in $C_0^{0,s-\frac{N}{m_0}}(\overline{\Omega})$ and the same reasoning is valid for $\{v_{p_n}\}$, showing that $\{v_{p_n}\}_{n \in \mathbb{N}}$ is uniformly bounded in $C_0^{0,t-\frac{N}{m_0}}(\overline{\Omega})$.

Passing to subsequences if necessary, there exist $u_\infty \in C_0^{0,s-\frac{N}{m_0}}(\overline{\Omega})$ and $v_\infty \in C_0^{0,t-\frac{N}{m_0}}(\overline{\Omega})$ such that
$$u_{p_n} \to u_\infty \quad \text{and} \quad v_{p_n} \to v_\infty \ \ \text{uniformly in} \ \ \Omega.
$$

We also observe that 
$$\Vert u_{\infty} \Vert_{\infty}^{\theta} \vert v_{\infty}(x_0) \vert^{1-\theta} = \lim_{n\to \infty}\left( \left(\int_{\Omega}\vert u_{p_n} \vert^{\alpha(p_n)} \dd x \right) \vert v_{p_n}(x_0) \vert^{\beta(p_n)} \right)^{\frac{1}{p_n}} = 1.
$$

Fix $k > \frac{N}{s}$. By applying Fatou's, Hölder's inequality and \eqref{upn}, we obtain
\begin{align}\label{equ4}
\int_{\Omega}\int_{\Omega} \frac{\vert u_\infty(x) - u_\infty(y)\vert^k}{\vert x- y \vert^{sk}} \dd x \dd y &\leq \liminf_{n \to \infty} \int_{\Omega}\int_{\Omega} \frac{\vert u_{p_n}(x) - u_{p_n}(y)\vert^k}{\vert x- y \vert^{\left(\frac{N}{p_n} + s\right)k}} \dd x \dd y \nonumber\\
&\leq \liminf_{n \to \infty} \vert \Omega \vert^{2\left(\frac{p_n-k}{p_n}\right)} \left(\int_{\Omega}\int_{\Omega} \frac{\vert u_{p_n}(x) - u_{p_n}(y)\vert^{p_n}}{\vert x- y \vert^{N + sp_n}} \dd x \dd y \right)^{\frac{k}{p_n}}\nonumber\\
&\leq \vert \Omega \vert^2 \liminf_{n \to \infty} [u_{p_n}]_{s,p_n}^{k} \\
&\leq \vert \Omega \vert^2 \liminf_{n \to \infty} \left(p_n^{\frac{1}{p_n}}\sqrt[p_n]{\Lambda_1(p_n)}\right)^k \nonumber\\
&\leq \vert \Omega \vert^2 \left(\frac{1}{R^{s\theta + (1-\theta)t}}\right)^k \nonumber.
\end{align}

Thus,
$$\vert u_\infty \vert_s = \lim_{k\to \infty} \left(\int_{\Omega}\int_{\Omega} \frac{\vert u_\infty(x) - u_\infty(y)\vert^k}{\vert x- y \vert^{sk}} \dd x \dd y\right)^{\frac{1}{k}} \leq \lim_{n \to \infty} \vert \Omega \vert^{\frac{2}{k}}\,\frac{1}{R^{s\theta + (1-\theta)t}} = \frac{1}{R^{s\theta + (1-\theta)t}}.
$$

Analagously, 
$$\vert v_\infty \vert_t = \lim_{k\to \infty} \left(\int_{\Omega}\int_{\Omega} \frac{\vert v_\infty(x) - v_\infty(y)\vert^k}{\vert x- y \vert^{tk}} \dd x \dd y\right)^{\frac{1}{k}} \leq \lim_{n \to \infty} \vert \Omega \vert^{\frac{2}{k}}\,\frac{1}{R^{s\theta + (1-\theta)t}} = \frac{1}{R^{s\theta + (1-\theta)t}}
$$
and therefore
$$\max\big\{\vert u_\infty \vert_s, \vert v_\infty \vert_t \big\} \leq \frac{1}{R^{s\theta + (1-\theta)t}}.
$$

It follows from Lemma \ref{prop1} that
$$\frac{1}{R^{s\theta + (1-\theta)t}} = \inf_{(u,v) \in X_{s,t,\infty}^*(\Omega)}\frac{\max\big\{\vert u \vert_s, \vert v \vert_t \big\}}{\Vert u \Vert_{\infty}^{\theta} \vert  v(x_0) \vert^{1-\theta}} \leq  \max\big\{\vert u_\infty \vert_s, \vert v_\infty \vert_t \big\} \leq \frac{1}{R^{s\theta + (1-\theta)t}},
$$
thus producing
\[
\max\big\{\vert u_\infty \vert_s, \vert v_\infty \vert_t \big\} = \frac{1}{R^{s\theta + (1-\theta)t}}.
\]

On its turn, inequality \eqref{equ4} yields
\begin{align*}
\max\left\{ \left(\int_{\Omega}\int_{\Omega} \frac{\vert u_\infty(x) - u_\infty(y)\vert^k}{\vert x- y \vert^{sk}} \dd x \dd y\right)^{\frac{1}{k}},  \left(\int_{\Omega}\int_{\Omega} \frac{\vert v_\infty(x) - v_\infty(y)\vert^k}{\vert x- y \vert^{tk}} \dd x \dd y\right)^{\frac{1}{k}} \right\}
\end{align*}
\begin{align*} \leq \vert \Omega \vert^\frac{2}{k} \liminf_{n \to \infty}\left( p_n^{\frac{1}{p_n}}\sqrt[p_n]{\Lambda_1(p_n)}\right).
\end{align*}
Thus, as $k \to \infty$ we obtain
\begin{align*}
\frac{1}{R^{s\theta + (1-\theta)t}} = \max\big\{\vert u_\infty \vert_s, \vert v_\infty \vert_s \big\}&\leq \liminf_{n\to \infty}\left( p_n^{\frac{1}{p_n}}\sqrt[p_n]{\Lambda_1(p_n)}\right)\\ &\leq \limsup_{n\to \infty}\left( p_n^{\frac{1}{p_n}}\sqrt[p_n]{\Lambda_1(p_n)}\right) \leq \frac{1}{R^{s\theta + (1-\theta)t}},
\end{align*}
from what follows
$$\lim_{n\to \infty} \sqrt[p_n]{\Lambda_1(p_n)} = \lim_{n\to \infty}\left( p_n^{\frac{1}{p_n}}\sqrt[p_n]{\Lambda_1(p_n)}\right) = \frac{1}{R^{s\theta + (1-\theta)t}} = \Lambda_{1,\infty}.
$$

\vspace*{-.7cm}$\hfill\Box$

\section{Proof of Theorem \ref{thm4}}\label{proofthm4}
The next result only shows that solutions in the weak sense are viscosity solutions. Its proof can be achieved by adapting the arguments given by Lindgren and Lindqvist in \cite[Proposition 1]{Lind}. 
\begin{proposition}\label{visc}
The function $u_p$ e $v_p$ given by Theorem \ref{thm1} are viscosity solutions to the problems
\begin{equation*}
\left\{\begin{array}{ll}
\mathcal{L}_{s,p}u = \Lambda_1(p) \alpha(p) \vert u \vert^{\alpha(p)-1}v(x_0) & {\rm in} \ \ \Omega,\\
u= 0 & {\rm in} \ \mathbb{R}^N\setminus\Omega,
\end{array}\right.\end{equation*}
and
\begin{equation*}
\left\{\begin{array}{ll}
\mathcal{L}_{t,p}v = 0 & {\rm in} \ \Omega \setminus \{x_0\},\\
v= 0 & {\rm in} \ \mathbb{R}^N\setminus\Omega,\\
v(x_0) = v_p(x_0),
\end{array}\right.
\end{equation*}
respectively.
\end{proposition}

\noindent\textit{Proof of Theorem \ref{thm4}.} We start showing that $v_\infty$ is a viscosity solution to the problem
\begin{equation}\label{Pvis4}
\left\{\begin{array}{ll}
\mathcal{L}_{t,\infty} v = 0 & {\rm in} \ \ \Omega \setminus \{x_0\},\\
v  = 0 & {\rm in} \ \mathbb{R}^N\setminus\Omega,\\
v(x_0) = v_\infty(x_0).
\end{array}
\right.
\end{equation}

According to Theorem \ref{thm3} we have $v_\infty =0$ in $\mathbb{R}^{N}\setminus \Omega$ and $v_\infty(x_0) = v_\infty(x_0)$. So, we need only show that $v_\infty$ is a viscosity solution. Fix $(z_0,\varphi) \in \left(\Omega\setminus \{x_0\}\right) \times C_0^{1}(\mathbb{R}^{N}\setminus\{x_0\})$ satisfying
$$\varphi(z_0) = v_\infty(z_0) \qquad\text{and} \qquad \varphi(x) \leq v_\infty(x), \ \ \forall x \in \mathbb{R}^{N}\setminus\{x_0,z_0\}.
$$

Theorem \ref{thm3} also guarantees the existence of a sequence $\{(u_{p_n}, v_{p_n})\}_{n \in \mathbb{N}} \in C_{0}^{0,s}(\overline{\Omega}) \times C_{0}^{0,t}(\overline{\Omega})$ such that $u_{p_n} \to u_\infty$ and $v_{p_n} \to v_\infty$ uniformly in $\Omega$. Thus, there exists a sequence $\{x_{p_n}\}_{n \in \mathbb{N}}$ so that $x_{p_n} \to z_0$ and $v_{p_n}(x_{p_n}) = \varphi(x_{p_n})$. Since $x_0 \neq z_0$, we can assume the existence of $n_0 \geq 0$ and a ball $B_{\rho}(z_0)$ such that
$$x_{p_n} \notin B_{\rho}(z_0) \subset \Omega\setminus\{z_0\}, \quad \forall n \geq n_0.
$$

Since $v_{p_n}$ weakly satisfies 
$$(-\Delta_{p_n})^{t}v_{p_n}(x) =  \Lambda_{1}(p_n) \alpha(p_n) \left(\int_{\Omega}\vert u_{p_n} \vert^{\alpha(p_n)} \dd x \right) \vert v_{p_n}(x_0) \vert^{\beta(p_n)} v_{p_n}(x_0)\delta_{x_0}$$
in $\Omega$, then also in $\Omega\setminus \{x_0\}$, Proposition \ref{visc} yields that $v_{p_n}$ is a viscosity solution to the problem
\begin{equation} \label{Pvis2}
\left\{\begin{array}{llll}
\mathcal{L}_{t,p_n}v = 0 & {\rm in} \ \ \Omega \setminus \{x_0\},\\
v= 0 & {\rm in} \ \mathbb{R}^N\setminus\Omega,\\
v(x_0) = v_{p_n}(x_0).
\end{array}
\right.
\end{equation}

By standard arguments, we obtain a sequence $\{z_n\}_{n \in \mathbb{N}} \subset B_{\rho}(x_0)$ such that $z_n \to z_0$ and
$$\sigma_n := \min_{B_{\rho}(x_0)}\left( v_{p_n} - \varphi \right) = v_{p_n}(z_n) - \varphi(z_n) < v_{p_n}(x) - \varphi(x), \ \ \forall x \neq x_{p_n}.
$$

Now, define $\Psi_n:= \varphi + \sigma_n$. We have
$$\Psi_n(z_n) = \varphi(z_n) + \sigma_n =  v_{p_n}(z_n) \qquad \text{and} \qquad \Psi_n(x) = \varphi(x) + \sigma_n < v_{p_n}(x), \ \ \forall x \in B_{\rho}(x_0).
$$

Since $v_{p_n}$ satisfies \eqref{Pvis2} in $\Omega\setminus \{x_0\}$, 
$$(\mathcal{L}_{t,\infty}\Psi_n)(z_n) \leq 0, \qquad \forall n \geq n_0.
$$

Thus, defining
$$\left(A_{p_n,t}(\varphi(z_n))\right)^{p_n-1}:= 2\int_{\mathbb{R}^N} \frac{\vert \varphi(z_n) - \varphi(y) \vert^{p_n-2}(\varphi(z_n) - \varphi(y))^{+}}{\vert z_n-y \vert^{N+tp_n}} \dd y
$$
and
$$\left(B_{p_n,t}(\varphi(z_n))\right)^{p_n-1}:= 2\int_{\mathbb{R}^N} \frac{\vert \varphi(z_n) - \varphi(y) \vert^{p_n-2}(\varphi(z_n) - \varphi(y))^{-}}{\vert z_n-y \vert^{N+tp_n}} \dd y,
$$
we have
\begin{align}\label{Lim}
\left(A_{p_n,t}(\varphi(z_n))\right)^{p_n-1} - \left(B_{p_n,t}(\varphi(z_n))\right)^{p_n-1} &= 2\int_{\mathbb{R}^N} \frac{\vert \varphi(z_n) - \varphi(y) \vert^{p_n-2}(\varphi(z_n) - \varphi(y))}{\vert z_n-y \vert^{N+sp_n}} \dd y\nonumber\\  &\leq 0 , \quad \forall n \geq n_0.
\end{align}

Applying \cite[Lemma 3.9]{GGA} (see also \cite[Lemma 6.1]{Ferreira}), we obtain
$$\lim_{n \to \infty} A_{p_n,t}(\varphi(z_n)) =  \left(\mathcal{L}_{t,\infty}^{+}\varphi\right)(z_0) \qquad \text{and} \qquad \lim_{n \to \infty} B_{p_n,t}(\varphi(z_n)) =  \left(-\mathcal{L}_{t,\infty}^{-}\varphi\right)(z_0).
$$

As $n \to \infty$ in $\eqref{Lim}$ we get
$$\left(\mathcal{L}_{t,\infty}\varphi\right)(x_0) = \left(\mathcal{L}_{t,\infty}^{+}\varphi\right)(x_0) + \left(\mathcal{L}_{t,\infty}^{-}\varphi\right)(x_0) \leq 0,
$$
showing that $v_\infty$ is a viscosity supersolution of \eqref{Pvis4}. Analogously, we obtain that $v_\infty$ is a viscosity subsolution of the same equation, and thus a viscosity solution of \eqref{Pvis4}.

Now we show that $u_\infty$ is a viscosity solution to the problem
\begin{equation} \label{Pvis5}
\left\{\begin{array}{ll}
\max\bigg\{\mathcal{L}_{s,\infty} u, \mathcal{L}^{-}_{s,\infty} u + \Lambda_{1,\infty} \vert u(x) \vert^{\theta} \vert v_\infty(x_0)\vert^{1-\theta}\bigg\}= 0 & {\rm in} \ \ \Omega,\\
u = 0 & {\rm in} \ \mathbb{R}^N\setminus\Omega.
\end{array}
\right.
\end{equation}

The same reasoning used before imply that, for given  $(z_0,\varphi) \in \Omega \times C_0^{1}(\mathbb{R}^{N})$, we find a sequence $\{u_{p_n}\}_{n \in \mathbb{N}}$ in $C_{0}^{0,s}(\overline{\Omega})$ such that $u_{p_n} \to u_\infty$ uniformly in $\Omega$ and a sequence $\{x_{p_n}\}_{n \in \mathbb{N}}$ satisfying $x_{p_n} \to z_0$ and $u_{p_n}(x_{p_n}) = \varphi(x_{p_n})$. Thus, there exist $n_0 \geq 0$ and a ball $B_{\rho}(z_0)$ so that
$$x_{p_n} \notin B_{\rho}(z_0) \subset \Omega\setminus\{z_0\}, \ \ \forall n \geq n_0.
$$

As before, we obtain that $u_{p_n}$ is a viscosity solution to the problem
\begin{equation*}
\left\{\begin{array}{ll}
\mathcal{L}_{s,p_n}u_{p_n} = \Lambda_1(p_n) \alpha(p_n) \vert u_{p_n} \vert^{\alpha(p_n)-1}v_{p_n}(x_0) & {\rm in} \ \ \Omega,\\
u= 0 & {\rm in} \ \mathbb{R}^N\setminus\Omega.
\end{array}
\right.
\end{equation*}

Considering, as before, a sequence $\{z_n\}_{n \in \mathbb{N}} \subset B_{\rho}(z_0)$ such that $z_n \to z_0$ and defining $\Psi_n$ as in the previous proof, we obtain
$$(\mathcal{L}_{s,p_n}\Psi_n)(z_n) \leq \Lambda_1(p_n) \alpha(p_n) \vert \Psi_n(z_n) \vert^{\alpha(p_n)-1}v_{p_n}(x_0)  \ \ \forall n \geq n_0,
$$
which is equivalent to the inequality 
\begin{equation*}
\left(A_{p_n,s}(\varphi(z_n))\right)^{p_n-1} - \left(B_{p_n,s}(\varphi(z_n))\right)^{p_n-1} \leq \left(C_{p_n}(\varphi(z_n))\right)^{p_n-1}  \ \ \forall n \geq n_0,
\end{equation*}
where
$$\bigg(C_{p_n}(\varphi(z_n))\bigg)^{p_n-1} := \Lambda_1(p_n) \alpha(p_n) \vert \varphi + \sigma_n \vert^{\alpha(p_n)-1}v_{p_n}(x_0)
$$
and the other terms are analogous to that of the previous case, just changing $t$ for $s$.

Observe that a direct calculation yields
\begin{align*}
\lim_{n \to \infty} C_{p_n}(\varphi(z_n)) &= \lim_{n \to \infty} \left( \sqrt[p_n]{\Lambda_1(p_n)} \sqrt[p_n]{\alpha(p_n)} \vert \varphi(z_n) + \sigma_n \vert^{\frac{\alpha(p_n)}{p_n-1}}v_{p_n}(x_0)^{\frac{\beta(p_n)}{p_n-1}}\right)\\
&= \Lambda_{1,\infty} \vert\varphi(z_0)\vert^{\theta} v_{\infty}(x_0)^{1-\theta}
\end{align*}

So, as $n \to \infty$ em $\eqref{Lim}$ we obtain
$$\left(\mathcal{L}_{s,\infty}\varphi\right)(x_0) = \left(\mathcal{L}_{s,\infty}^{+}\varphi\right)(z_0) + \left(\mathcal{L}_{s,\infty}^{-}\varphi\right)(z_0) \leq \Lambda_{1,\infty} \vert\varphi(z_0)\vert^{\theta} v_{\infty}(x_0)^{1-\theta}
$$
and therefore
$$\max\left\{\mathcal{L}_{s,\infty} u, \mathcal{L}^{-}_{s,\infty} u - \Lambda_{1,\infty} \vert u(x) \vert^{\theta} \vert v_\infty(x_0)\vert^{1-\theta}\right\} \leq  0 \ \ {\rm in} \ \ \Omega,$$
that is, $u_\infty$ is a viscosity supersolution to problem \eqref{Pvis4}. Analogously, $u_\infty$ is a viscosity subsolution to the same problem. We are done.
$\hfill\Box$

\begin{remark}\label{variation}
We observe that the system 
\begin{equation} \label{Pinfty}\tag{$P^1_\infty$}
\left\{\begin{array}{ll}
(-\Delta_p)^{s}u(x) = \lambda \alpha(p) \vert u \vert^{\alpha(p)-2} u \vert v(x_v)\vert^{\beta(p)}  & {\rm in} \ \ \Omega,\\
(-\Delta_p)^{t}v(x) = \lambda \beta(p) \left(\displaystyle\int_{\Omega}\vert u \vert^{\alpha(p)} \dd x\right) \vert v(x_v) \vert^{\beta(p)-2} v(x_v) \delta_{x_v} & {\rm in} \ \ \Omega,\\
u= v=0 & {\rm in} \ \mathbb{R}^N\setminus\Omega,\\
\end{array}\right.
\end{equation}
where $x_v$ is a maximum point of $v$ in $\overline{\Omega}$ can be treated in the same setting given in Section \ref{setting}, applying the same procedure used to solve system \eqref{P1}.
\end{remark}
\section{On the system\eqref{P2}}\label{functional}
In this section we consider the functional system \eqref{P2}.
\begin{equation*}
\left\{\begin{array}{ll}
(-\Delta_p)^{s}u(x) = \lambda \alpha(p) \vert u(x_1) \vert^{\alpha(p)-2} u(x_1) \vert v(x_2) \vert^{\beta(p)} \delta_{x_1} & {\rm in} \ \ \Omega,\\
(-\Delta_p)^{t}v(x) = \lambda \beta(p) \vert u(x_1) \vert^{\alpha(p)} \vert v(x_2) \vert^{\beta(p)-2} v(x_2) \delta_{x_2} & {\rm in} \ \ \Omega,\\
u= v=0 & {\rm in} \ \mathbb{R}^N\setminus\Omega,
\end{array}
\right.
\end{equation*}
where $x_1,x_2\in\Omega$ are arbitrary points, $x_1\neq x_2$. Observe that both equations are functional, so their treatment recall that used to deal with the second equation in system \eqref{P1}.

\begin{definition}
A pair  $(u,v) \in X_{s,t,p}(\Omega)$ is a weak solution to \eqref{P2} if
\begin{align}\label{defisolp2}
\left\langle (-\Delta_p)^s u, \varphi \right\rangle + \left\langle (-\Delta_p)^s v, \psi \right\rangle = \lambda&\left[\alpha(p)\vert u(x_1)\vert^{\alpha(p)-2}u(x_1)\vert v(x_2) \vert^{\beta(p)}\varphi(x_1)\right. \\
&\left.\ + \beta(p)\vert u(x_1)\vert^{\alpha(p)}\vert v(x_2) \vert^{\beta(p)-2}v(x_2) \psi(x_2) \right]\nonumber
\end{align}
for all $(\varphi,\psi) \in X_{s,t,p}(\Omega)$.
\end{definition}

The denominator in the definition of $Q_{s,t,p}$ should be changed into $|u(x_1)|^{\alpha(p)}\,|v(x_2)|^{\beta(p)}$, maintaining the definition of $\Lambda_1(p)$. The first result, which is similar to Theorem \ref{thm1} is the following.
\begin{theorem}\label{thm12}
For each $p \in \left(\frac{N}{s}, \infty\right)$ we have
\begin{enumerate}
\item[$(i)$] $\Lambda_1(p) > 0$;
\item[$(ii)$] there exist $(u_p,v_p) \in X^{*}_{s,t,p}(\Omega)$ such that $u_p > 0$, $v_p > 0$ and
$$\vert u_p(x_1) \vert^{\alpha(p)}\vert v_p(x_2) \vert^{\beta(p)} = 1 \qquad\text{and}\qquad\Lambda_{1}(s,p) = Q_{s,t,p}(u_p,v_p).	
$$
\end{enumerate}
\end{theorem}
Its proof is also similar to that of Theorem \ref{thm1}. For details, see the proof sketched in Section \ref{remarks} or \cite[Theorem 1]{MJRD}. 

The next step is to prove a result similar to Theorem \ref{thm2}. Changing the definition of $S_{p}$ and $S{\infty}$ into
\begin{align*}S_{p} &= \left\{(u,v) \in X_{s,t,p}(\Omega)\,:\, \vert u(x_1)\vert^{\alpha(p)}\vert  v(x_2) \vert^{\beta(p)} = 1\right\}
\intertext{and} 
S_{\infty} &= \left\{(u,v) \in X_{s,t,p}\,:\, \vert u(x_1)\vert^{\theta}\vert  v(x_2) \vert^{1-\theta} = 1\right\}
\end{align*}
and also the denominator in $G_p$ into $\vert u(x_1)\vert^{\theta}\vert  v(x_2) \vert^{1-\theta}$, we obtain the version of Theorem \ref{thm2} with the same statement.

Up to this point, the points $x_1,x_2\in \Omega$ were taken arbitrarily. Now, we consider sequences $u_n:=u_{p_n}$ and $v_n:=u_{p_n}$ given by Theorem \ref{thm1}. Since $u_n,v_n>0$, we can take $x_1$ as a maximum $x_n$ of $u_n$ and $x_2$ as a maximum $y_n$ of $v_n$. Observe that we do not suppose that the maxima $x_n$ and $y_n$ are unique. However, we will prove that the sequence $(x_n,y_n)$ has a subsequence that converges to $(x_\infty,y_\infty)$ and the equality $|u_\infty(x_\infty)|^\theta|v_\infty(y_\infty)|^{1-\theta}=1$ still holds true. 
\begin{theorem}\label{thm32}
Let  $\{p_n\}$ be a sequence converging to $\infty$ and $(u_{p_n}, v_{p_n})$ the solution of \eqref{P1} given in  Theorem \ref{thm12}. Denote $x_n:=x_{u_{p_n}}$ and $y_n:=x_{v_{p_n}}$ a sequence of maxima to $u_{p_n}$ and $v_{p_n}$, respectively. Passing to a subsequence if necessary,  $\{(u_{p_n}, v_{p_n})\}_{n \in \mathbb{N}}$ converges uniformly to $(u_\infty, v_\infty) \in C_{0}^{0,s}(\overline{\Omega}) \times C_{0}^{0,s}(\overline{\Omega})$, while the sequences $\{x_n\}$ and $\{y_n\}$ converge to $x_{\infty}\in\Omega$ and $y_\infty\in\Omega$, respectively, which are the maxima of $u_\infty$ and $v_\infty$.  Furthermore
\begin{enumerate}
\item [$(i)$] $u_\infty \geq 0$, $v_\infty \geq 0$ and $\vert u_\infty(x_\infty) \vert^{\theta} \vert v_\infty(y_\infty) \vert^{1 - \theta}=1$;
\item [$(ii)$] $\displaystyle\lim_{n\to \infty} \sqrt[p_n]{\Lambda_1(p_n)} = \frac{1}{R^{s\theta+(1-\theta)t}}$
\item [$(iii)$] $\max\big\{\vert u_\infty \vert_s, \vert v_\infty \vert_t \big\} = \displaystyle\frac{1}{R^{s\theta+(1-\theta)t}}$;
\item [$(iv)$] If $s=t$, then $$0 \leq u_\infty(x) \leq \displaystyle\frac{\left(\textup{dist}(x, \mathbb{R}^{N}\setminus \Omega)\right)^{s}}{R^{s}}\quad\text{and}\quad 0 \leq v_\infty(x) \leq \frac{\left(\textup{dist}(x, \mathbb{R}^{N}\setminus \Omega)\right)^{s}}{R^{s}}.$$
\end{enumerate}
\end{theorem}
Its proof can be obtained by mimicking the method used to prove Theorem \ref{thm3}. Comparing this result with the one in \cite{MJRD}, we first note that our result brings information about the sequence of maxima of $u_{p_n}$ and $v_{p_n}$, which are absent in that paper. 

Finally, the analogue to Theorem \ref{thm4} is the following. Once again, its proof is obtained by adapting that of the Theorem \ref{thm4}. 
\begin{theorem}The functions $u_\infty$ and $v_\infty$, given by Theorem \ref{thm32}, are viscosity  solutions of the problems
\begin{equation*}
\left\{\begin{array}{ll}
\mathcal{L}_{s,\infty} u = 0 & {\rm in} \ \ \Omega \setminus \{x_1\},\\
u = 0 & {\rm in} \ \mathbb{R}^N\setminus\Omega,\\
u(x_1) = u_\infty(x_1)
\end{array}\right.\qquad\text{and}\qquad
\left\{\begin{array}{ll}
\mathcal{L}_{t,\infty} v = 0 & {\rm in} \ \ \Omega \setminus \{x_2\},\\
v = 0 & {\rm in} \ \mathbb{R}^N\setminus\Omega,\\
v(x_2) =v_\infty(x_2),
\end{array}\right.
\end{equation*}
respectively.
\end{theorem}


\begin{thebibliography}{10}
\bibitem{GGC} C. Alves, G. Ercole and G. Pereira: Asymptotic behavior as $p\to\infty$ of ground state solutions of a $(p,q(p))$-Laplacian problem, \textit{Proc. Roy. Soc. Edinburgh Sect. A}, \textbf{149} (2019), no. 6, 1493--1522.
	
\bibitem{Chambolle} A. Chambolle, E. Lindgren and R. Monneau: A Hölder infinity Laplacian, \textit{ESAIM Control Optim. Calc. Var.} \textbf{18} (2012), no. 3, 799--835.

\bibitem{Del}L. Del Pezzo and J. Rossi: Eigenvalues for systems of fractional $p$-Laplacians, \textit{Rocky Mountain J. Math.} \textbf{48} (2018), no. 4, 1077--1104. 


\bibitem{guia}R. Di Nezza, G. Palatucci and E. Valdinoci: Hitchhikers guide to the fractional Sobolev spaces, \textit{Bull. Sci. Math.} \textbf{136} (2012), no. 5, 521--573.

\bibitem{GG} G. Ercole and G. Pereira: Asymptotics for the best Sobolev constants and their extremal functions, \textit{Math. Nachr.} \textbf{289} (2016), no. 11--12, 1433--1449.
	
\bibitem{GGR} G. Ercole, G. Pereira and R. Sanchis: Asymptotic behavior of extremals for fractional Sobolev inequalities associated with singular problems, \textit{Ann. Mat. Pura Appl. (4)}, \textbf{198} (2019), no. 6, 2059--2079.
	
\bibitem{GGA} G. Ercole,  A. H. S. Medeiros and G. A. Pereira: On the behavior of least energy solutions of a fractional $(p, q (p))$-Laplacian problem as $p$ goes to infinity, \textit{Asymptot. Anal.} \textbf{123} (2021), no. 3--4, 237--262.
	
\bibitem{Ferreira} R. Ferreira and M. Pérez-Llanos: Limit problems for a fractional $p$-Laplacian as $p\to \infty$, \textit{NoDEA Nonlinear Differential Equations Appl.} \textbf{23} (2016), no. 2, Art. 14, 28 pp. 
	
\bibitem{Lind}E. Lindgren and P. Lindqvist: Fractional eigenvalues. \textit{Calc. Var. Partial Differential Equations} \textbf{49} (2014), no. 1--2, 795--826.
	
\bibitem{Lind2} P. Juutinen and P Lindqvist: On the higher eigenvalues for the $\infty$-eigenvalue problem, \textit{Calc. Var. Partial Differential Equations} \textbf{23} (2005), no. 2, 169--192. 
	
	
\bibitem{MJRD} M. Mih\v{a}ilescu, J. Rossi and D. Stancu-Dumitru: A limiting problem for a family of eigenvalue problems involving p-Laplacians, \textit{Rev. Mat. Complut.} \textbf{32} (2019), no. 3, 631--653.
	
	
\end{thebibliography}
\end{document}